\documentclass{amsart}

\usepackage{amsfonts}
\usepackage{amsmath}
\usepackage{amsthm}
\usepackage{amssymb}
\usepackage{latexsym}
\usepackage{enumerate}

\newtheorem{theorem}{Theorem}[section]
\newtheorem{lemma}[theorem]{Lemma}
\newtheorem{proposition}[theorem]{Proposition}
\newtheorem{corollary}[theorem]{Corollary}

\newtheorem{definition}[theorem]{Definition}

\numberwithin{equation}{section}

\begin{document}

\newcommand{\bb}{\mathfrak{b}}
\newcommand{\cc}{\mathfrak{c}}
\newcommand{\N}{\mathbb{N}}
\newcommand{\R}{\mathbb{R}}
\newcommand{\A}{{\mathbb{R}}_+\cup\{0\}}
\newcommand{\forces}{\Vdash}
\newcommand{\LL}{\mathbb{L}}
\newcommand{\K}{\mathbb{K}}
\newcommand{\Q}{\mathbb{Q}}

%Antonio's commands
\newcommand{\To}{\longrightarrow}

\title{A Banach space in which every injective operator is surjective}

\author{Antonio Avil\'{e}s}
\thanks{First author by was supported by MEC and FEDER (Project MTM2011-25377), {\em Ramon y Cajal} contract (RYC-2008-02051) and an FP7-PEOPLE-ERG-2008 action.}
\email{avileslo@um.es}
\address{Departamento de Matem\'{a}ticas, Universidad de Murcia, 30100 Murcia (Spain)}

\author{Piotr Koszmider}
\thanks{The second author was partially supported by the National Science Center research grant DEC-2011/01/B/ST1/00657. } 
\email{P.Koszmider@Impan.pl}
\address{Institute of Mathematics, Polish Academy of Sciences,
ul. \'Sniadeckich 8,  00-956 Warszawa, Poland}

%
%    General info
%
%\date{January 1, 1994 and, in revised form, June 22, 1994.}
%
%\keywords{Differential geometry, algebraic geometry}
%
\begin{abstract} 
We construct an infinite dimensional Banach space of continuous functions $C(K)$ such that every one-to-one operator $T:C(K)\To C(K)$ is onto.
\end{abstract}

\maketitle

\markright{}

\section{Introduction}

Already S. Banach has asked if every infinite dimensional Banach space $X$ has a proper subspace 
isomorphic to $X$ (\cite{banach}). 
It took many decades until this problem has been solved in the negative by T. Gowers in
\cite{gowershyperplane}. Thus, in spaces $X$ like in \cite{gowershyperplane} all isomorphisms $T: X\rightarrow X$
must be onto $X$. This property is shared by spaces with few operators  constructed and investigated
for example in \cite{gowersmaurey} or \cite{argyrostolias} 
based on hereditarily indecomposable Banach spaces (abbreviated HI) as well
as in Banach spaces of the form $C(K)$ with few operators first constructed in \cite{few}.
In this paper answering a question from \cite{spectrum} we strengthen this property proving the following:

\begin{theorem}
There is an infinite dimensional Banach space $X$ such that whenever a bounded linear operator
$T: X\rightarrow X$ is injective, then it is an isomorphism onto $X$. Moreover $X$
is of the form $C(K)$.
\end{theorem}

Our approach is to construct a Banach space  of the form $C(K)$ with few operators
in the sense of \cite{few}. Additional special properties of the compact space present in our construction 
($K$ is an almost $P$-space) guarantee that every injective operator is an isomorphism
and so previously known arguments for $C(K)$ spaces with few operators can be used to conlude that the operator is onto. This seems the
only possible approach of modifying known spaces with few operators that can work.
Indeed,
separable Banach spaces and HI spaces cannot serve for 
our purposes because the main property may be shared only by finite dimensional 
spaces or those which have nonseparable dual space in the weak$^*$ topology (see \cite{spectrum}):

\begin{proposition} Suppose $X$ is an infinite dimensional  Banach space
whose dual ball which is separable in the weak$^*$ topology. Then
there is an injective 
operator on $X$ which is not  an automorphism of $X$. In particular this property is shared by
all separable infinite dimensional Banach spaces and all HI spaces.
\end{proposition}
\begin{proof} By a result from \cite{goldberg}, if the dual ball of $X$ is separable in
the weak$^*$ topology, then there is a compact 
injective operator $T: X\rightarrow X$. A compact operator cannot be onto an infinite dimensional Banach space.  
As  any HI space embeds into $\ell_\infty$ (A. IV. 6 of \cite{argyrosramsey}) whose dual
ball is separable in the weak$^*$ topology, we conclude that such spaces have separable dual
balls as well.
\end{proof}
We also prove  (see \ref{necessity}, \ref{pspaceccc} and the remarks after it)
that a Banach space of the form $C(K)$ where all injective operators are automorphisms cannot  
be one of the known constructions as in \cite{few, plebanek, big, fajardo, grande}.

Another remark is that our space satisfies a version of the invariant subspace property for nonseparable spaces, namely every operator $T:X\To X$ has an invariant subspace $Y$ such that both $Y$ and $X/Y$ are of density $\mathfrak{c}$. Indeed, the subspace $Y$ can be chosen to be complemented and such that still every injective operator $Y\To Y$ or $X/Y\To X/Y$ is an isomorphism. This is analogous to what happens in some separable Banach spaces with few operators~\cite{argyroshaydon,argyrosmotakis}, where it can be proven that every operator has a proper infinite-dimensional invariant subspace. 

We wish to express our gratitude to Amin Kaidi. The discussion with him during a visit of the first author to the University of Almeria is in the origin of this paper. 

\section{Almost $P$-spaces and weak multipliers}

If $K$ is a compact and Hausdorff and $g\in C(K)$ we can define
an operator $T_g:C(K)\rightarrow C(K)$ by $T_g(f)=fg$, likewise for
$g$ a Borel bounded function on $K$ we can define ${_{g^*}T}: M(K)\rightarrow M(K)$
by $\int f d{_{g^*}T}(\mu)=\int fgd\mu$ for every $\mu\in M(K)$ and $f\in C(K)$.

\begin{definition}
Let $T:C(K)\rightarrow C(K)$ be a linear bounded operator.
We say that $T$ is a weak multiplication if and only if there
is a $g\in C(K)$ and a weakly compact operator $S:C(K)\rightarrow C(K)$
such 
$$T=T_g+ S.$$
We say that $T$ is a weak multiplier if and only if there
is a Borel  $g^*:K\rightarrow K$ and a weakly compact operator $S:C^*(K)\rightarrow C^*(K)$
such that 
$$T^*={_{g^*}T}+ S.$$
We will say that a space $C(K)$ has few operators if and only if
all operators on $C(K)$ are weak multiplications or weak multipliers.
\end{definition}

Of course a weak multiplication is  a weak multiplier and the form
of the weak multiplication is much nicer. However, 
the notion of a weak multiplier plays a more natural role for example having all
operators weak multipliers is invariant under isomorphisms of $C(K)$ spaces  and having all operators weak
multiplications is not (\cite{schlackow}). 

It is easy to construct a weak multiplier on $C(K)$ which is not a weak
multiplication if we have a function which is discontinuous at one point and
the discontinuity cannot be removed by changing the value in this point. For example
$\chi_{[0,{1\over 2}]}$ is such a function on $[0,1]$ and we get
$T:C([0,1])\rightarrow C([0,1])$ defined by $$T(f)=\chi_{[0,{1\over 2}]} (f-f({1\over {2}})),$$
then $T^*={_{g^*}T}$
where $g^*=\chi_{[0,{1\over 2}]}$ is the Borel 
function which is discontinuous at one point  ${1\over 2}$. 
We say that an $x\in K$ is $C^*$-embedded if and only every continuous 
$g: K\setminus\{x\}\rightarrow [0,1]$ can be extended to a  continuous
function on $K$. It is proved in 2.7 of \cite{few} that every weak multiplier is a weak multiplication
on $C(K)$ if and only if for every $x\in K$ the subspace $K\setminus\{x\}$ is $C^*$-embedded in $K$.

\begin{definition}[\cite{levy}] A topological space is called an almost $P$-space if and only if 
every nonempty $G_\delta$-set has non-empty interior.
\end{definition}

\begin{proposition}\label{necessity} If $K$ is not a weak $P$-space, then there is a an injective
operator $T: C(K)\rightarrow C(K)$ which is not an automorphism of $C(K)$.
\end{proposition}
\begin{proof} Let $U_n\subseteq K$ be open subsets of $K$ for $n\in\N$ such that
$\bigcap_{n\in\N} U_n=F\not=\emptyset$ has empty interior. We may assume that
${\overline{U_{n+1}}}\subseteq U_n$ for each $n\in\N$, and so that $F$ is closed.
 Let $g_n: K\rightarrow [0, 1/2^n]$ be 
 continuous functions satisfying $g_n(x)=1/2^n$ whenever $x\in K\setminus U_n$ and
$g_n(x)=0$ whenever $x\in F$. Let $g=\Sigma_{n\in \N}g_n$. 
Note that $g\in C(K)$ assumes value $0$ in $x\in K$
if and only if $x\in F$. Consider the operator $T_g$. The operator is injective, because any nonzero $f\in C(K)$
has values separated from zero on an open set, i.e., not included in $F$. $T_g$ is not an isomorphism
because whenever $n\in \N$ and the support of $f$ is inluded in $U_n$  we have $||T_g(f)||\leq (1/2^{n-1})||f||$.
It follows from the open mapping theorem that the
image of $T_g$ is not closed, in particular, $T_g$ is not onto $C(K)$.
\end{proof}

Recall that a topological space satisfies the countable chain condition
if and only if it does not admit an uncountable pairwise disjoint collection of open sets.

\begin{lemma}\label{pspaceccc} No infinite compact Hausdorff almost $P$-space $K$ satisfies the countable chain condition.
\end{lemma}
\begin{proof}
We may assume that $K$ has at most countably many
isolated points. This excludes scattered spaces where isolated points form
a dense open set whose complement would be $G_\delta$ under the above assumption.
 In any nonscattered compact space we can construct a family
$\{U_s: s\in \{0,1\}^n, n\in \N\}$ of nonempty open sets such that $U_s\subseteq U_t$
whenever $t\subseteq s$ and $U_s\cap U_t=\emptyset$ whenever $s\cup t$ is not a function
where $s, t\in  \{0,1\}^n$ and $n\in \N$. If $K$ is a weak $P$-space for every
 $\sigma\in\{0,1\}^\N$ there is nonempty open $U_\sigma\subseteq U_s$ for all 
$s\in \{0,1\}^n$ with $s\subseteq \sigma$. It follows that $\{U_\sigma: \sigma\in \{0,1\}^\N\}$
is a pairwise disjoint family of nonempty open sets of $K$ of cardinality continuum.
\end{proof}

The compact spaces constructed in \cite{few, plebanek, big, fajardo, grande} all 
satisfy the  countable chain condition (which is essentially used), and so they are not
almost $P$-spaces. It follows that the constructions like in \cite{few, plebanek, big, fajardo, grande}
cannot be used to conclude the main result of this paper.

The crucial properties of $P$-spaces in the context of multiplications and
weakly compact operators are the following:

\begin{lemma}\label{weaklycpspace} Suppose that $K$  is a compact Hausdorff space which
does not satisfy the countable chain condition. Then no weakly compact
operator $T: C(K)\rightarrow C(K)$ is injective.
\end{lemma}
\begin{proof} Suppose that $T:C(K)\rightarrow C(K)$ is weakly compact. 
Then $||T(f_n)||\rightarrow 0$ for any sequence of pairwise disjoint (i.e., $f_n^.f_m=0$ for distinct $n,m\in\N$)
functions in $C(K)$ (\cite{diesteluhl}). Let $\{U_i: i\in I\}$ be an uncountable pairwise disjoint family of nonempty
open subsets of $K$ existing by Lemma \ref{pspaceccc} and let $f_i\in C(K)$
 be nonzero functions with supports included in
$U_i$ for each $i\in I$. It follows that $||T(f_i)||=0$, that is $T(f_i)=0$, for
all but countably many $i\in I$, and therefore $T$ is not injective.
\end{proof}

\begin{lemma}\label{injectiveisomorphism} Suppose that $K$ is a compact Hausdorff which is an almost $P$-space
and $g\in C(K)$. $T_g$ is an isomorphism onto its range if and only if $T_g$ is injective.
\end{lemma}
\begin{proof} If $T_g$ is not an isomorphism onto its range, then $g(x_0)=0$ for
some $x_0\in K$. Since $\{x\in K: g(x)=0\}$ is $G_\delta$, it follows that it contains a
nonempty open set $U\subseteq K$. Taking a nonzero function $f\in C(K)$ with its support
in $U$, we have $T_g(f)=0$ and hance $T_g$ is not injective.
\end{proof}

\begin{proposition}\label{Pspaceweakmultiplier} Suppose that $K$ is a compact totally disconnected almost $P$-space such that
 every operator on $C(K)$ is a weak multiplier. Then every injective linear bounded operator
on $C(K)$ is an automorphism of $C(K)$.
\end{proposition}
\begin{proof} Let $T$ be an injective operator on $C(K)$. Since
we assume that all operators on $C(K)$ are weak multipliers, by Theorem 2.2. (b) of \cite{few} 
there is a bounded borel $g:K\rightarrow \R$
with at most countably many points of discontinuity and a weakly compact operator
$S:C(K)^*\rightarrow C(K)^*$ such that $T^*={_{g^*}T}+S$. Let $X=\{x_n: n\in \N\}$
be the set of the points of discontinuity of $g$.

First we will prove   that there is no infinite sequence $Y=\{y_n: n\in \N\}$ of points of $K$ such
that $|g(y_n)|<1/n$. Supose that there is such a $Y$ and let us
derive a contradiction. Let $y\in K\setminus X$ be an accumulation point of $Y$.
It exists as the set of accumulation points of $Y$ is closed and with no isolated points since $K$ cannot have 
nontrivial convergent sequences because this would give a 
complemented in $C(K)$ copy of $c_0$ yielding a noncommutativity of
the quotent ring of all operators on $C(K)$ divided by weakly compact operators,
contradicting by 4.5. of \cite{schlackow} the fact that all operators on $C(K)$ are weak multipliers. Moreover $g(y)=0$ since $g$ is continuous at $y$.

Consider  $V_n=g^{-1}[(-1/n, 1/n)]\setminus X$. It is a relative
open set in $K\setminus X$. Let $U_n$ be an open set of $K$ such that $U_n\setminus X=V_n$.
Let $U=\bigcap(U_n\setminus\{x_k:k<n\})$. $U\ni y$ is a non-empty $G_\delta$ set, and
so it has a nonempty interior $W$. We may assume that $W$ is clopen.
 Of course $g|_W=0$. Taking a function $f\in C(K)$
with support in $W$ and any $\mu\in M(K)$ we have 
$$T^*(\mu)(f)=\int gfd\mu+S(\mu)(f)=0+S(\mu)(f).$$
So the dual of the restriction of $T$ to functions with supports in $W$ is weakly compact and so
 by Gantmacher theorem the restriction of $T$ to functions with supports in $W$ is weakly compact.
In particular $T$ is not injective by Lemma \ref{weaklycpspace} as $W$
is an almost $P$-space. It follows that $T$ is
not injective, a contradiction. This completes the proof of the nonexistence
of $Y$ as above.

Hence there is $k\in \N$ such that $|g(x)|>1/k$ for all but finitely many points $x\in K$.
By modifying $g$ to $\hat{g}$ on this finite set of points we obtain $T^*=_{\hat{g}^\ast}T+S'$ where
$|\hat{g}(x)|>1/k$ for all points $x\in K$ and $S'=S+_{(g-\hat{g})^\ast}T$ is weakly compact on $M(K)$. 
$_{\hat{g}^\ast}$T is now an automorphism on $M(K)$, $S'$ is strictly singular as weakly compact
on a dual to $C(K)$ by \cite[VI.8.10 and p. 394]{dunfordschwartz}.
By Propositon 2 (c) 10 of \cite{LT} $T^*=_{\hat{g}^\ast}T+S'$ is Fredholm and so has a closed range
and has the same Fredholm index as $_{\hat{g}^\ast}T$ that is zero. 
As $T$ is injective, $T^*$ is onto, hence as an operator of
Fredholm index zero $T^*$ has zero kernel  an by the open mapping theorem it 
is an automorphism of $M(K)$. Hence
$T$ is an automorphism of $C(K)$ by Ex. 2.39 of \cite{fabianetal} which completes the proof of the proposition.
\end{proof}

In this paper we shall construct a compact space $K$ that satisfies the hypotheses of Proposition~\ref{Pspaceweakmultiplier}. We do not know if one can prove without extra set-theoretic axioms 
that such a $K$ can exist so that every operator on $C(K)$ is actually a weak multiplication. In fact the countable chain condition was needed in the constructions like in
\cite{few, plebanek, big, fajardo, grande} to ensure that 
$K\setminus\{x\}$ is $C^*$-embedded in $K$ for every $x\in K$ and
consequently (by 2.7 of \cite{few}) that every weak multiplier is actually a weak multiplication.
It should be noted however, that assuming an additional set-theoretic axiom $\diamondsuit$ (see \cite{jech, kunen})
such construction seems possible. 

Let us say that a Boolean algebra is an almost $P$-algebra if its Stone space is an almost $P$-space. We will use the symbols $\wedge$, $\vee$, $-$ and $\leq$ to denote the usual operations and the order in a Boolean algebra. If $a$ is an element of a Boolean algebra, we will denote by $[a]$ the corresponding clopen subset in its Stone space.

\begin{lemma}\label{sumlemma}
Suppose that $\mathcal A$ is a Boolean algebra which is an almost $P$-algebra.
Suppose that $f_n\in C(K_{\mathcal A})$ for each $n\in \N$ and
$\sum_{n\in \N} f_n(x)>\delta$ for some $x\in K_{\mathcal A}$ and
some $\delta\in \R$. Then there is $a\in \mathcal A$ such that
$$\sum_{n\in \N} f_n(y)>\delta$$
for every $y\in [a]$.
\end{lemma}
\begin{proof}
By the continuity of $f_n$, for each $n\in \N$ there is
an $a_n\in \mathcal A$ such that $x\in [a_n]$ and $f_n(y)$ is close to $f_n(x)$
for each $y\in [a_n]$. So $\bigcap_{n\in \N} [a_n]$ is a nonempty $G_\delta$ set and 
so has a nonempty interior as $K_{\mathcal A}$ is an almost $P$-space. It follows
that there is $a\in \mathcal A$ as required.
\end{proof}

\section{Separation by submorphisms instead of separation by suprema}

We say that two subsets $I$ and $J$ of a Boolean algebra $\mathcal{A}$ are orthogonal if $x\wedge y = 0$ for all $x\in I$, $y\in J$. We say that $I$ and $J$ are separated in $\mathcal{A}$ if there exists $c\in\mathcal{A}$ such that  $x\leq c$, $y\wedge c = 0$ for all $x\in I$, $y\in J$.

We will say that a function between Boolean algebras $\phi:\mathcal{A}_1\To \mathcal{A}_2$ is a submorphism if $\phi(a\wedge b) = \phi(a)\wedge \phi(b)$, $\phi(a\vee b) = \phi(a)\vee \phi(b)$ and $\phi(0)=0$. If it also satisfies $\phi(1)=1$, then we call $\phi$ a morphism.\\

\begin{theorem}\label{generalization}
Let $\mathcal{A}$ be a Boolean algebra which is an almost $P$-algebra and $K_{\mathcal A}$ its Stone space. Suppose that for every pairwise disjoint family $a = \{a_n: n <\omega\}\subset \mathcal{A}\setminus\{0\}$ there exists a submorphism $\phi_a:\mathcal{P}_a\To \mathcal{A}$ such that
\begin{enumerate}
\item $\mathcal{P}_a$ is a subalgebra of $\mathcal{P}(\omega)$ that contains all finite sets.
\item $\phi_a(\{n\}) = a_n$
\item For any infinite $\sigma\subset \omega$, and for any family of pairwise disjoint $$\{b_n , n\in \sigma\}\subset \mathcal{A}\setminus\{0\}$$
orthogonal with $\{a_n: n\in \N\}$, there exists $\tau\subset\sigma$ such that
\begin{enumerate}
\item $\tau\in\mathcal{P}_a$
\item $\{b_n : n\in \tau\}$ and $\{b_n : n\in\sigma\setminus\tau\}$ are not separated in $\mathcal{A}$.
\end{enumerate}
\end{enumerate}
Then every operator $T:C(K_{\mathcal A})\To C(K_{\mathcal A})$ is a weak multiplier.\\
\end{theorem}
\begin{proof}
 Suppose that $\mathcal A$ is a as above and that $T:C(K_{\mathcal A})
\rightarrow C(K_{\mathcal A})$ is not a weak multiplier. By 2.1 and 2.2 of \cite{few} 
there is an $\varepsilon>0$,  an antichain $\{a_n: n\in \N\}$ of $\mathcal A$ and points $x_n\in K$ such that
$x_n\not \in [a_n]$ and 
$$|T(1_{A_n})(x_n)|>\varepsilon.\leqno 1)$$
By a slight modification
of the points $x_n$ we may assume that they are not in the closed and nowhere dense set
$F={\overline{\bigcup_{n\in \N}[a_n]}}\setminus\bigcup_{n\in \N}[a_n]$.

By applying  the Ramsey theorem \cite[Theorem 9.1]{jech}
to a coloring $c: [\N]^2\rightarrow \{0,1\}$ where $c(n,m)=1$ iff $x_n\in [a_m]$ or $x_m\in [a_n]$
and the fact that $\{a_n: n\in \N\}$ is an antichain we may assume, by taking an infinite
 subset of $\N$  that $x_n\not\in [a_m]$
for any two distinct $n,m\in \N$.

By applying the Rosenthal lemma \cite[Lemma 1.1]{rosenthal} by taking an infinite subset of $\N$ we may assume that
$\Sigma_{n\in \N\setminus\{k\}} |T(1_{[a_n]})(x_k)|<\varepsilon/3$ holds for all $k\in \N$.
Of course if $|T(1_{[a_n]})(x_k)|<\delta$, for some $\delta>0$, 
 then there is  $b^k_n\not=0$ in $\mathcal A$ such that $x_k\in [b^k_n]$
and $||T(1_{[a_n]})|[b^k_n]||<\delta$. Since $\mathcal A$ is an almost $P$-algebra
there are $0\not=b_k\leq b_n^k$ for all $k\in \N$, so  by the choice of $x_n$ outside of $F$ we may assume 
that $\{b_n: n\in \N\}$ and $\{a_n: n\in \N\}$ are orthogonal and
$$\Sigma_{n\in \N\setminus\{k\}} ||T(1_{[a_n]})|b_k||<\varepsilon/3.\leqno 2)$$
Moreover, by 1) and considering a bit smaller $b_k$s we may assume that 
for all $n\in \N$ and for all $x\in [b_n]$ we have 
$$|T(1_{a_n})(x)|>\varepsilon.\leqno 3)$$
Let $\phi_a$ and $\mathcal P_a$ be as in the statement  of the Theorem
for $a=\{a_n: n\in \N\}$.

\noindent{\bf Case 1.} There is an infinite $\sigma\subseteq\N$ and
$0\not=c_k\leq b_k$ for all $k\in \N$ such that  
$$T(1_{\phi_a(\tau)})|[c_k]=\sum_{n\in \tau} T(1_{[a_n]})|[c_k]$$
holds for every $\tau\subseteq\sigma$, $\tau\in\mathcal{P}_a$.
\vskip 6pt
\noindent In this case, by 2) and 3), if $k\in \tau$, for every $x\in [c_k]$ we have
$$|T(1_{\phi_a(\tau)})(x)|=|\sum_{n\in \tau^1} T(1_{[a_n]})(x)|>2\varepsilon/3$$
and if $k\not\in \tau$, for every $x\in [c_k]$ we have
$$|T(1_{\phi_a(\tau)})(x)|=|\sum_{n\in \tau^1} T(1_{[a_n]})(x)|<\varepsilon/3$$
Then
$(T(1_{\phi_a(\tau)}))^{-1}((-\infty, \varepsilon/3])$ and $(T(1_{\phi_a(\tau)}))^{-1}([2\varepsilon/3, \infty))$
 separate $\{c_n: n\in \tau\}$ from $\{c_n: n\not\in \tau\}$ whenever $\tau\subset\sigma$, $\tau\in\mathcal{P}_a$. This contradicts  
condition (3) of the Theorem.
\vskip 6pt
\noindent{\bf Case 2.} Case 1 does not hold.
\vskip 6pt
\noindent Assuming the negation of the condition from case 1, we will carry out
a transfinite inductive construction  which will contradict the boundedness
of the operator $T$. 
Let $\{\sigma_\xi: \xi<\omega_1\}$ be an almost disjoint (i.e., such that pairwise intersections
of its elements are finite) family 
of infinite subsets of $\N$. 
For all $\xi<\omega_1$ construct:
\begin{itemize}
\item an infinite $\tau_\xi\subseteq \sigma_\xi$,
\item an antichain $\{c^\xi_k: k\in \N\}$  such that $0\not=c_k^{\xi'}\leq c_k^{\xi}\leq b_k$ for all $k$
and all $\xi\leq \xi'<\omega_1$,
\item $n_\xi, m_\xi\in \N\setminus\{0\}$
\end{itemize}
such that for all $x\in [c^\xi_{n_\xi}]$ we have
$$|T(1_{\phi_a(\tau_\xi)})(x)-\sum_{n\in \tau_\xi} T(1_{a_n})(x)|>1/m_\xi.\leqno (4)$$

The possiblity of such a construction follows from the fact that $K_{\mathcal A}$ is an almost $P$-space
Lemma \ref{sumlemma} and the assumption that Case 1 does not hold.

A single pair $(n', m')$ has appeared infinitely many times as $(n_\xi, m_\xi)$. Let $i\in \N$ be 
such that $i/6m'>||T||$ and consider $\xi_1<...< \xi_i$ such that $(n_{\xi_i}, m_{\xi_i})=(n', m')$.
Let $n_0\in \N$ be such that the pairwise intersections of all $\tau_{\xi_1}, ..., \tau_{\xi_i}$
be included in $\{0, ..., n_0-1\}$.  Puting $\tau_\xi'=\tau_\xi\setminus \{0, ..., n_0-1\}$ and noting that 
$$1_{\phi_a(\tau_\xi)}=1_{\phi_a(\tau_\xi')}+\sum_{n\in \tau_\xi\cap n_0} 1_{a_n}$$
we conclude from (4) that 
$$|T(1_{\phi_a(\tau_{\xi_j}')})(x)-\sum_{n\in \tau_{\xi_j}'} T(1_{a_n})(x)|>1/m'$$
for all $x\in [c^{\xi_i}_{n'}]$ and all $1\leq j\leq i$. 
By taking $n_0'>n_0$  and using (2) we may assume that
$|\sum_{n\in \tau_{\xi_j}'} T(1_{a_n})(x)|<1/3m'$ and so
$|T(1_{\phi_a(\tau_{\xi_j}')})(x)|>2/3m'$ for
all $x\in [c^{\xi_i}_{n'}]$ for all  $1\leq j\leq i$ and consequently for some  $F\subseteq \{1, ..., i\}$ 
of cardinality not smaller than $i/2$ and some $x\in [c^{\xi_i}_{n'}]$ we have
$$|\sum_{j\in F}T(1_{\phi_a(\tau_{\xi_j}')})(x)|>(1/3m')(i/2)>||T||.$$
However $\tau_\xi'$s are pairwise disjoint, so $\phi_a(\tau_\xi')$s are pairwise disjoint.
It follows that $||\sum_{j\in F}1_{\phi_a(\tau_\xi')}||=1$ and so
$||T(\sum_{j\in F}1_{\phi_a(\tau_{\xi_j}')})||\leq ||T||$, a contradiction.

\end{proof}

\emph{Remarks.} Theorem~\ref{generalization} holds true even without the assumption that $\mathcal{A}$ is an almost $P$-algebra. The proof in the general case would follow the steps of \cite[Theorem 3.3.2]{iryna}, which corresponds to the case in which $\mathcal{P}_a$ is the algebra of all subsets $\tau\subset\omega$ such that $\{a_n : n\in\tau\}$ has a supremum in $\mathcal{A}$, and $\phi_a(\tau) = \bigvee\{a_n : n\in\tau\}$. We decided to include a self-contained proof, and the assumption that $\mathcal{A}$ is a $P$-agebra simplifies the argument and it is enough for our purposes.

\begin{lemma}\label{keeppromise}
Let $\hat{\mathcal{A}}\subset \hat{\mathcal{B}}$ be Boolean algebras such that $|\hat{\mathcal{A}}|<\mathfrak{c} \leq |\hat{\mathcal{B}}|$, let $\phi:\mathcal{P}(\omega)\To \hat{\mathcal{B}}$ be a submorphism such that $\phi\{n\}\in \hat{\mathcal{A}}$ for every $n<\omega$, let $I$ be a set with $|I|<\mathfrak c$ and suppose that for every $i\in I$, $X_i,Y_i\subset \hat{\mathcal{A}}$ are orthogonal sets which are nonseparated in $\hat{\mathcal{A}}$. Then, there exists an infinite set $\tau\subset \omega$ such that $X_i$ nd $Y_i$ remain unseparated in $\hat{\mathcal{A}}\langle \phi\upsilon\rangle$ for all $\upsilon\subset \tau$ and all $i\in I$. 
\end{lemma}

\begin{proof}
We follow a similar argument as in \cite[p. $165_9$]{few}. Let $\mathfrak{M}$ be an almost disjoint family of subsets of $\omega$ with $|\mathfrak{M}|=\mathfrak{c}$. We shall prove that there exists $\tau\in\mathfrak{M}$ that satisfies the statement of the Lemma. We proceed by contradiction, so suppose that for every $\tau\in\mathfrak{M}$ there exist $i\in I$ and $\upsilon\subset \tau$ such that $Y_i$ and $X_i$ are separated in $\hat{\mathcal{A}}\langle\phi\upsilon \rangle$. 

 The separation means that there exist pairwise disjoint $b,c,d\in \hat{\mathcal{A}}$ such that 
$$z = (b\wedge \phi\upsilon ) \vee (c \wedge- \phi\upsilon ) \vee d$$ separates $Y_i$ and $X_i$.

Since $|\mathfrak{M}|=\mathfrak{c}$, and there are less than $\mathfrak{c}$ many choices for $i$, $b$, $c$ and $d$, it follows that there exists a single $i\in I$ and three fixed elements $b$, $c$ and $d$ such that for at least two different choices of $\tau\in\mathfrak{M}$ there exists $\upsilon\subset \tau$ such that
 $$z(\upsilon) = (b\wedge \phi\upsilon ) \vee (c \wedge - \phi\upsilon ) \vee d$$ separates $Y_i$ and $X_i$. 
So pick $\tau\neq \tau'$ in $\mathfrak{M}$, $\upsilon\subset \tau$, $\upsilon'\subset \tau'$ as above. Then, since both $z(\upsilon)$ and $z(\upsilon')$ separate $Y_i$ from $X_i$,
$$ z_0 = (b \wedge \phi\upsilon\wedge \phi\upsilon') \vee (c\wedge - (\phi\upsilon\wedge \phi\upsilon')) \vee d$$
also separates $Y_i$ from $X_i$. But $z_0\in\hat{\mathcal{A}}$ since $$\phi\upsilon \wedge \phi\upsilon' = \phi(\upsilon\cap \upsilon') = \bigvee_{n\in\upsilon\cap\upsilon'}\phi\{n\}$$
as $\upsilon\cap\upsilon'$ is finite. This is a contradiction with the hypothesis that $Y_i$ and $X_i$ are nonseparated in $\hat{\mathcal{A}}$.
\end{proof}

\section{The construction}

Along this section, $\mathcal{B}$ will be the Boolean algebra of subsets of $\N$ modulo finite sets, $\mathcal{B} = \mathcal{P}(\N)/Fin$.

\begin{theorem}
There exists an infinite Boolean algebra $\mathcal{A}\subset \mathcal{B}$ which satisfies the hypotheses of Theorem~\ref{generalization}.
\end{theorem}

\begin{proof}
We shall construct $\mathcal{A}$ as a union of a $\mathfrak{c}$-chain of subalgebras of $\mathcal{B}$, each subalgebra being of size less than $\mathfrak{c}$, in the form $\mathcal{A} = \bigcup_{\alpha<\mathfrak{c}}\mathcal{A}_\alpha$, $|\mathcal{A}_\alpha|<\mathfrak{c}$. The algebras $\mathcal{A}_\alpha$ are constructed by induction. For the first step, $\mathcal{A}_0$ can be any arbitrary infinite subalgebra of $\mathcal{B}$ of cardinality less than $\mathfrak c$.
Let us write $$\mathfrak{c} = \bigcup_{\alpha<\mathfrak{c}}F_\alpha \cup \bigcup_{\alpha<\beta<\mathfrak{c}}G_{\alpha\beta}$$
as a disjoint union where $|F_\alpha| = |G_{\alpha\beta}| = \mathfrak{c}$ and $F_\alpha \cap \alpha = G_{\alpha\beta}\cap \beta =\emptyset$ for each $\alpha<\beta<\mathfrak{c}$. We write $G_{\alpha\beta} = \{\xi_{\alpha\beta}^{\delta\gamma} : \delta,\gamma<\mathfrak c\}$ in such a way that $(\delta,\gamma)\mapsto \xi_{\alpha\beta}^{\delta\gamma}$ is a bijection between $\mathfrak{c}\times\mathfrak{c}$ and $G_{\alpha\beta}$. Once the algebra $\mathcal{A}_\alpha$ is defined, we make the following enumerations:\\

\begin{enumerate}
\item $\{(d_n^\xi)_{n\in \N} : \xi\in F_\alpha\}$ are all strictly decreasing sequences of $\mathcal{A}_\alpha$.\\

\item $\{\{a_n^\gamma(\alpha) : {n\in \N}\}: \gamma<\mathfrak{c}\}$ are all pairwise disjoint families in $\mathcal{A}_\alpha\setminus\{0\}$. We fix morphisms $\psi^\gamma_\alpha: \mathcal{P}(\omega)\To \mathcal{B}$ such that $\psi^\gamma_\alpha(\{n\}) = a_n^\gamma(\alpha)$. Such morphisms can be the following way:
let $\sigma_n\subseteq \N$ be pairwise disjoint sets such that $h(\sigma_n)=a_n$ where
$h:\wp(\N)\rightarrow \wp(\N)/Fin=\mathcal{B}$ is the canonical surjective homomorphism; now 
define  $\psi^\gamma_\alpha(\rho)=h[\bigcup_{n\in\rho}\sigma_n]$.\\

\item $\{\{b_n^\delta(\alpha)): {n\in\sigma^\delta(\alpha)}\} : \delta<\mathfrak{c}\}$ are all possible families where $\sigma^\delta(\alpha)$ is an infinite subset of $\N$ and $\{b_n^\delta(\alpha) : n\in\sigma^\delta(\alpha)\}$ is a family of pairwise disjoint elements of $\mathcal{A}_\alpha\setminus\{0\}$.\\

\end{enumerate}

At limit steps we shall define $\mathcal{A}_\alpha = \bigcup_{\beta<\alpha}\mathcal{A}_\beta$. At a successor stage, $\mathcal{A}_{\xi+1}$ will be the algebra generated by $\mathcal{A}_\xi$ and a certain element $x_\xi$ that we will add, $\mathcal{A}_{\xi+1} = \mathcal{A}_\xi\langle x_\xi\rangle$. When $\xi \in F_\alpha$ for some $\alpha$, $x_\xi$ will be a lower bound of $\{d_n^\xi : n<\omega\}$. When $\xi = \xi^{\gamma\delta}_{\alpha\beta}\in G_{\alpha\beta}$ for some $\alpha<\beta$ and the family $\{a_n^\gamma(\alpha),b_n^\delta(\beta), n\in\sigma^\delta(\beta)\}$ is pairwise disjoint, we shall find an infinite $\tau^\xi\subset\sigma^\delta(\beta)$ such that $B_+^\xi = \{b_n^\delta(\beta) : n\in\tau^\xi\}$ and $B_-^\xi=\{b_n^\delta(\beta) : n\in\sigma^\delta(\beta)\setminus \tau^\xi\}$ are nonseparated in $\mathcal{A}_\alpha$ and we shall define $x_\xi = \psi_\alpha^\gamma(\tau^\xi)$, $\mathcal{A}_{\xi+1} = \mathcal{A}_\xi\langle x_\xi\rangle$. In this way, in the steps in $
 F_\alpha$ we take care that we obtain an almost $P$-space, while in the steps in $G_{\alpha\beta}$ we take care that the hypotheses of Theorem~\ref{generalization} are satisfied\footnote{When $a=\{a_n^\gamma(\alpha) : n<\omega\}$, we will have $\mathcal{P}_a = \{x\subset\omega : \psi^\gamma_\alpha(x)\in\mathcal{A}\}$ and $\phi_a$ will be the restriction of $\psi^\gamma_\alpha$ to $\mathcal{P}_a$.}. In order this to work we must make sure that when we construct $\mathcal{A}_{\xi+1}$, those families of the form $B^{\xi'}_+$ and $B^{\xi'}_-$ that were chosen to be nonseparated in some previous step $\xi'<\xi$, remain nonseparated in $\mathcal{A}_{\xi+1}$, assuming inductively that they were kept nonseparated in $\mathcal{A}_\xi$. More precisely, this is done as follows:\\

\begin{enumerate}

\item If $\xi\in G_{\alpha\beta}$ for some $\alpha<\beta\leq\xi$, then $\xi = \xi^{\gamma\delta}_{\alpha\beta}$ for some $\gamma,\delta<\mathfrak{c}$. We have $\{a_n = a_n^\gamma(\alpha) : n<\omega\}$ and $\psi=\psi^\gamma_\alpha:\mathcal{P}(\omega)\To\mathcal{B}$ on the one hand, and $\{b_n = b_n^\delta(\beta) : n\in\sigma = \sigma^\delta(\beta)\}$ on the other hand. If $a_n\wedge b_m \neq\emptyset$ for some $n,m\in\sigma$ we do nothing and just define $\mathcal{A}_{\xi+1} = \mathcal{A}_\xi$ (we call this a trivial step). Otherwise, we can apply Lemma~\ref{keeppromise} for $\hat{\mathcal{B}} = \mathcal{B}$, $\hat{\mathcal{A}} = \mathcal{A}_\xi$, $\phi = \psi|_{\mathcal{P}(\sigma)}: \mathcal{P}(\sigma)\To \mathcal{B}$, and the sets $B^\zeta_+$ and $B^\zeta_-$ which are kept nonseparated in $\mathcal{A}_\xi$ for $\zeta<\xi$ by the inductive hypothesis.
Let $\tau\subset \sigma$ be given by Lemma~\ref{keeppromise}. Since $|\mathcal{A}_\xi|<\mathfrak c$ there are less than $\mathfrak c$ many $\upsilon\subset \tau$ such that $\{b_n : n\in\upsilon\}$ and $\{b_n : n\in\sigma\setminus\upsilon\}$ are separated. So choose $\tau^\xi\subset \tau$ such that $\{b_n : n\in\upsilon\}$ and $\{b_n : n\not\in\sigma\setminus\upsilon\}$ are not separated, $x_\xi = \psi(\tau^\xi)$ and $\mathcal{A}_{\xi+1} = \mathcal{A}_\xi\langle x_\xi\rangle$.\\

\item If $\xi\in F_\alpha$ for some $\alpha\leq\xi$, then $\{d_n^\xi : n<\omega\}$ is a strictly decreasing sequence of elements of $\mathcal{A}_\alpha\subset \mathcal{A}_\xi$, and we must add an element below it. If there exists some element $a\in\mathcal{A}_\xi$ such that $a< d_n^\xi$ for all $n$, then we do not need to add anything, and we make just $\mathcal{A}_{\xi+1} = \mathcal{A}_\xi$. Otherwise, choose $x_\xi = x\in\mathcal{B}\setminus\{0\}$ such that $x<d_n^\xi$ for all $n$. Notice that if $y\in \mathcal{A}_\xi$ and $y\leq x$ then $y=0$.

%Let $U$ be an ultrafilter of $\mathcal{A}_\xi$ such that $\{d_n^\xi: n<\omega\}\subset U$. Since $|U|<\mathfrak{c}$, by the extension property of $\mathcal{B}$, we can find $x_\xi = x\in \mathcal{B}$ such that $x\subset u$ for all $u\in U$. Define $\mathcal{A}_{\xi+1} = \mathcal{A}_\xi\langle x\rangle$.\\ 
%
We have to check that $B^\zeta_+$ and $B^\zeta_-$ remain unseparated in $\mathcal{A}_{\xi+1}$ for $\zeta<\xi$. So suppose for contradiction that they were separated by some element $z = (b\wedge x) \vee (c\wedge- x) \vee d$ with $b,c,d\in \mathcal{A}_\xi$ pairwise disjoint. We claim that $z' = (c\vee d)$ also separate $B^\zeta_+$ and $B^\zeta_-$, which is a contradiction since $z'\in\mathcal{A}_\xi$. On the one hand, if $u\in B^\zeta_+$ then $$u< z = (b\wedge x) \vee (c\wedge - x) \vee d < b\vee c \vee d.$$ But moreover, $u\wedge b< z\wedge b < x$ and since $u\wedge b \in\mathcal{A}_\xi$, we get that $u\wedge b = 0$. This proves that $u< z' = (c\vee d)$. On the other hand if $v\in B^\zeta_-$, then $v\wedge z = 0$.  This implies that $v\wedge c\wedge -e = 0$, so $v\wedge c< e$ and since $v\wedge c \in\mathcal{A}_\xi$, we get that $v\wedge c = 0$. It is clear that we also have $v\wedge d = 0$, so $v\wedge z' = v\wedge (c\vee d) = 0$ as required.   
\end{enumerate}
\end{proof}

\begin{corollary}
There exists a Banach space $X\subset \ell_\infty/c_0$ of the form $X=C(K)$ such that every injective operator $T:X\To X$ is surjective.
\end{corollary}

\section{Invariant subspaces}

\begin{proposition}
If $K=K_\mathcal{A}$ is as in Theorem~\ref{generalization}, then for every operator $T:X\To X$ there exists a proper nonempty clopen subset $L$ of $K$ such that $\{f\in C(K) : f|_{K\setminus L} = 0\}$ is an invariant subspace of $T$.
\end{proposition}

\begin{proof}
By Theorem~\ref{generalization}, $T$ is a weak multiplier, so $T^\ast = _{g^\ast}T+S$ where $g^\ast:K\To K$ is Borel and $S$ is weakly compact.It is enough to find a proper nonempty clopen $L$ such that the set $N_L$ of measures whose support is disjoint from $L$ is invariant under $T^\ast$. The set $N_L$ is invariant under $_{g^\ast}T$ for all clopens $L$ and all Borel functions $g$, hence it is enough to find $L$ for which $N_L$ is invariant under $S$. We prove that if $\{L_\alpha : \alpha<\omega_1\}$ is a disjoint family of nonempty clopen subsets of $K$, then there exists $\alpha$ such that $N_{L_\alpha}$ is invariant under $S$. If it was not the case, then for every $\alpha<\omega_1$ there exists a measure $\mu_\alpha\in N_{L_\alpha}$, that we can take with $\|\mu_\alpha\|=1$, such that $S\mu_\alpha|_{L_\alpha} \neq 0$. There is an $\varepsilon>0$ such that $|S\mu_\alpha|(L_\alpha)>\varepsilon$ for uncountably many $\alpha<\omega_1$. This contradicts weak compactness by the Dieudonné-Grothendieck theorem \cite[VII.14]{diestel}.
\end{proof}

The invariant subspace $Y=\{f\in C(K) : f|_{K\setminus L} = 0\}$ is complemented and isomorphic to $C(L)$, the complement being isomorphic to $C(K\setminus L)$. If $L$ a nonempty clopen set of $K$, then $L$ is still as in Theorem~\ref{generalization}, so again the Banach space $C(L)$ has the same properties as $C(K)$: $C(L)$ has density $\mathfrak{c}$ and every injective operator $C(L)\To C(L)$ is an isomorphism. But note that $C(L)$ is not isomorphic to $C(K)$, or otherwise we could construct an injective operator $C(K)\To C(K)$ which is not surjective.

\end{document}